\documentclass{amsart}
\usepackage{graphicx}
\usepackage{amssymb}
\newtheorem{thm}{Theorem}[section]

\newtheorem{lem}[thm]{Lemma}

\theoremstyle{definition}
\newtheorem{defn}[thm]{Definition}
\newtheorem{rem}[thm]{Remark}
\theoremstyle{question}
\newtheorem{que}[thm]{Question}
\theoremstyle{Conjecture}

\numberwithin{equation}{section}

\begin{document}

\title[on the norm of the centralizers of a group]{on the norm of the centralizers of a group}%
\author{Mohammad Zarrin}%

\address{Department of Mathematics, University of Kurdistan, P.O. Box: 416, Sanandaj, Iran}%
 \email{m.zarrin@uok.ac.ir, zarrin@ipm.ir}
\begin{abstract}For any group $G$, let $C(G)$ denote the intersection of the normalizers of centralizers of all
elements of $G$. Set $C_0= 1$. Define
$C_{i+1}(G)/C_i(G)=C(G/C_i(G))$ for $i\geq 0$. By $C_{\infty}(G)$
denote the terminal term of the ascending series. In this paper,
we show that a finitely generated group $G$ is nilpotent
 if and only if $G = C_{n}(G)$ for some positive integer $n$. 
\end{abstract}
\maketitle

\section{\textbf{ Introduction and results}}
For any group $G$, the norm $B_1(G)$ of $G$ is the intersection of
all the normalizers of subgroups of $G$ (in fact, $B_1(G)$ is the
intersection of all the normalizers of non-$1$-subnormal subgroups
of $G$). This concept was introduced by R. Baer in 1934 and was
investigated by many authors, for example, see \cite{Bae2,BHN,Sch}. It is
well-known \cite{Sch} that $Z(G)\leq B_1(G)\leq Z_2(G)$. More recently in \cite{zar} it has been generalized and
showed that the intersection of all the normalizers of
non-$n$-subnormal subgroups of $G$, say $B_n(G)$ (with the
stipulation that $B_n(G) = G$ if all subgroups of G are
$n$-subnormal) is a nilpotent normal subgroup of $G$ of class
$\leq \mu(n)$, where $\mu(n)$ is the function of Roseblade's
Theorem.

 The author in \cite{zar1} showed, in view of the proof of the main theorem, that every group with finitely many $n$ of
 centralizers is nilpotent-by-(finite of order $(n-1)!$). That is,
 $$\mid G/\bigcap_{a\in G} N_G(C_G(a))\mid\leq (n-1)!.$$
 (See Theorem 2.2 of \cite{zar2} and also Theorem B of \cite{zar3}.) This result suggests that the
behavior of centralizers has a strong influence on the structure
of the group (for more information see \cite{zar4} and \cite{zar5}). This  is the main motivation to introduce a new series of norms in groups by their normalizers of the centralizers.
\begin{defn}
For any group $G$, we define the subgroup $C(G)$ to be the
intersection of the normalizers of the centralizers of $G$. That
is, $$C(G)=\bigcap_{a\in G} N_G(C_G(a)).$$ Clearly $B_1(G)\leq
C(G)$. Define the series
whose terms $C_i(G)$ are characteristic subgroups as follows:\\
$C_{i+1}(G)/C_i(G)=C(G/C_i(G))$ for $i\geq 0$. By $C_{\infty}(G)$
denote the terminal term of the ascending series.
\end{defn}
We say that a group $G$ is a $\overline{C}_n$-group
($\overline{C}_\infty$-group) if $C_n(G) = G$ for some $n\in
\mathbb{N}$ ($G=C_\infty(G)$, respectively).\\

We give a characterization for finitely generated nilpotent groups in terms of the subgroups $C_i(G)$ of $G$, as follows:\\

\noindent{\bf Theorem.} Let $G$ be a finitely generated group.
Then the following statement are equivalent:
\begin{enumerate}
 \item $G$ is nilpotent;
\item $G = C_{n}(G)$ for some positive integer $n$; \item
$G/C_{m}(G)$ is nilpotent for some positive integer $m$.
 \end{enumerate}

 \section{\textbf{Proof} }

For the proof of the main Theorem we need the following Lemmas.\\
An element $x$ of $G$ is called right $n$-Engel if $[x,_n y]=1$
for all $y\in G$, where $[x, y] = x^{-1}y^{-1}xy=x^{-1}x^y$ and
$[x,_{m+1} y] = [[x,_{m} y], y]$ for all positive integers $m$. We
denote by $R_n(G)$, the set of all right $n$-Engel elements of $G$
and for a given positive integer $n$, a group is called $n$-Engel
if $G=R_{n}(G)$.
\begin{lem}\label{l}
For any group $G$, the subgroup $C(G)$ is nilpotent of class $\leq
3$ and so it is soluble of class $\leq 2$.
\end{lem}
\begin{proof}
Let $x\in C(G)$. Then, by definition of $C(G)$, $C_G(y)^x=C_G(y)$,
for all $y\in G$. It follows that $[x,_2y]=1$, for all $y\in
C(G)$. That is, $C(G)$ is a $2$-Engel group. But it is well-known
that every $2$-Engel group is a nilpotent group of class at most
3, completing the proof.
\end{proof}
 \begin{rem}\label{r}
 Since $C(G)$ is a nilpotent group of class $\leq 3$, it is easy to see that every
  $\overline{C}_n$-group is a soluble group of class at most $2n$.
 \end{rem}
The converse of the above Remark is not true in general. For
example the symmetric group of degree $3$, $S_3$ is not a
$\overline{C}_1$-group.\\
Here we show that the class of $\overline{C}_1$-groups is closed
by subgroups. In fact, we have.
 \begin{lem}\label{l1}
For every subgroup $H$ of $G$, we have
 $$H\cap C(G)\leq C(H).$$
\end{lem}
\begin{proof}
We have $H\cap C(G)=H\cap (\bigcap_{a\in G}
N_G(C_G(a)))=\bigcap_{a\in G} (H\cap N_G(C_G(a)))=\bigcap_{a\in G}
N_H(C_G(a))\leq \bigcap_{a\in H} N_H(C_G(a))\leq \bigcap_{a\in H}
N_H(C_H(a))=C(H)$ which is our assertion.
 \end{proof}
   We denote by $Z_{i}(G)$ is
$i$-term of the ascending central series of
  $G$. Here we give a very close connection between this series and the upper central
  series.
\begin{lem}\label{l2}
For any group $G$, we have $$Z_{i+1}(G)\leq C_{i}(G)\subseteq
R_{2i}(G).$$
\end{lem}
\begin{proof}
 We let
$C_i=C_{i}(G)$, and proceed by induction on $i$. First we show
that $Z_{i+1}(G)\leq C_{i}(G)$. It is clear, if $i =1$. Assume
that $x\in Z_{i+1}(G)$. So $[x,y]\in Z_{i}(G)$ for all $y\in G$
and so, by the induction hypothesis, $[x,y]\in C_{i-1}$. It
follows that $y^x=yt$ for some $t\in C_{i-1}$ and therefore
$$C_{G/C_{i-1}}(y^xC_{i-1})=C_{G/C_{i-1}}(yC_{i-1}).$$
Which implies that $x\in C_i(G)$. Hence $Z_{i+1}(G)\leq
C_{i}(G)$.\\
Now we show that $C_{i}(G)\subseteq R_{2i}(G)$. Again, it is
clear, if $i =1$. Assume that $x\in C_{i+1}(G)$. It follows that
$[x,_2y]\in C_{i}(G)$ for all $y\in G$. So, by the induction
hypothesis, $[x,_2y]\in R_{2i}(G)$ and so $x\in R_{2i+2}(G)$, and
this completes the proof.
 \end{proof}
 We note that it is not true in general that
$Z_{i+1}(G)=C_i(G)$. For
 instance, if $G$ is the dihedral group of size 32, then $Z_3(G)<C_2(G)=Z_{4}(G)=G$. In fact, we have the following Lemma.
 \begin{lem}\label{l5}
Let $G$ be a dihedral group of degree $n$, $D_{n}$. Then
$$C_i(G)=Z_{2i}(G),$$ for any $i\geq 0$.
\end{lem}
\begin{proof}
Suppose that  $n=2^{\alpha}m$, where $(2,m)=1, \alpha\geq 0$. It
is easy to see that
$$|C_1(D)|=\left\{%
\begin{array}{ll}
    1 & \hbox{$\alpha \leq 1$;} \\
    2 & \hbox{$\alpha =2 $;} \\
    4 & \hbox{$\alpha \geq 3$.} \\
\end{array}%
\right.$$ It follows, by Lemma \ref{l2}, that $Z_2(G)\leq C_1(G)$
and so $Z_2(G)=C_1(G)$. Hence $C_i(G)=Z_{2i}(G)$ (note that
$Z_j(G)/Z_i(G)=Z_{i+j}(G)/Z_i(G)$ for any $i, j\geq 0$.)
 \end{proof}

The class of nilpotent groups is not closed under forming extensions.
However, we have the following well-known result, due to P. Hall
(this result is often very useful for proving that a group is nilpotent).\\

\noindent{\bf{Theorem (P. Hall).}}
Let $N$ be a normal subgroup of a group $G$. If $G/N'$ and $N$
are nilpotent, then $G$ is nilpotent.

Here we show that the following statement (note that the subgroup $C(G)$ is nilpotent).
\begin{lem}\label{l3}
For any finitely generated group $G$, we have  $$G/C(G) \text
{~~is~~ nilpotent}\Longleftrightarrow G\text {~~is~~ nilpotent}.$$
\end{lem}
\begin{proof}
Let $G/C(G)$ is a finitely generated nilpotent group and $x\in G$.
By definition of $C(G)$, $N_G(C_G(x))/C(G)$ is a subgroup of
$G/C(G)$ and so, as $G/C(G)$ is nilpotent, it is subnormal
subgroup of $G/C(G)$. It follows that  $N_G(C_G(x))$ is subnormal
subgroup of $G$, written $N_G(C_G(x))\unlhd\unlhd G$. Hence
$$\langle x\rangle\unlhd C_G(x)\unlhd N_G(C_G(x))\unlhd\unlhd G.$$
Therefore $\langle x\rangle\unlhd \unlhd G$. That is, every cyclic
subgroup of $G$ is a subnormal subgroup of $G$. So $G$ is a
finitely generated Baer group, where a group $G$ is said to be Baer if for every $x\in G$ the cyclic subgroup
$\langle x\rangle$ is subnormal in $G$. But it is well-known that Baer
groups are locally nilpotent. Hence $G$ is a nilpotent group and
this completes the proof.
 \end{proof}
  \begin{lem}\label{l4}
Let $H$ is a subgroup of finitely generated group $G$. Then we
have
$$H/C_i(G) \text {~~is~~ nilpotent}\Longleftrightarrow H\text
{~~is~~ nilpotent}.$$
\end{lem}
\begin{proof}
We argue by induction on $i$. Let $H/C(G)$ is a finitely generated
nilpotent group. Then according to Lemma \ref{l1}, $C(G)=C(G)\cap
H\leq C(H)$ and so
$$H/C(H)\cong (H/C(G))/(C(H)/C(G)).$$ From which it follows that
$H/C(H)$ is nilpotent and so, by Lemma \ref{l3},  $H$ is
nilpotent. Now assume that $i>1$ and $H/C_{i+1}(G)$ is a nilpotent
group. In this case we have $$H/C_{i+1}(G)\cong
(H/C_i(G))/(C_{i+1}(G)/C_i(G)).$$ Applying the induction
hypothesis (note that $C_{i+1}(G)/C_i(G)=C(G/C_i(G))$ and
$H/C_i(G)\leq G/C_i(G)$), we thus conclude that $H/C_i(G)$ is
nilpotent and so, again by the induction hypothesis, $H$ is
nilpotent, completing the proof.
 \end{proof}
We can now deduce the main Theorem.\\

\noindent{\bf{Proof of the Theorem.}} According to Lemma
\ref{l2}, every nilpotent group of class $n+1$ is a
$\overline{C}_n$-group. Now assume that $G$ is a finitely
generated $\overline{C}_n$-group. According to Lemma \ref{l2} and
Remark \ref{r}, we conclude that $G$ is finitely generated
$2n$-Engel soluble group, so it is well-known (see \cite{Gru})
that $G$ is
nilpotent. Finally Lemma \ref{l4} completes the proof.\\

\begin{rem}
In view of Lemma \ref{l2}, one can see that every nilpotent group
(note necessarily finitely generated) of class $n+1$ is a
$\overline{C}_n$-group.
\end{rem}

Finally, we state the following Question.
\begin{que}
Is the nilpotency class of every nilpotent $\overline{C}_n$-group bounded
by $n$?
\end{que}

\end{document}